\documentclass[onefignum,onetabnum]{siamart190516}

\usepackage{lipsum}
\usepackage{amsfonts}
\usepackage{graphicx}
\usepackage{epstopdf}
\usepackage{algorithmic}

\ifpdf
  \DeclareGraphicsExtensions{.eps,.pdf,.png,.jpg}
\else
  \DeclareGraphicsExtensions{.eps}
\fi

\newsiamremark{remark}{Remark}
\newsiamremark{hypothesis}{Hypothesis}
\crefname{hypothesis}{Hypothesis}{Hypotheses}
\newsiamthm{claim}{Claim}

\headers{Computing monomial interpolating basis}{Y. H. Gong, X. Jiang, and B. X. Shang}

\title{Computing monomial interpolating basis for multivariate polynomial interpolation\thanks{Submitted to the editors DATE.
\funding{This work was supported  by National Natural Science Foundation of China under Grant No. 11671169 and 11901402.}}}

\author{Yihe Gong\thanks{College of Science, Northeast Electric Power University, Jilin, China
  (\email{yhegong@163.com}).}
\and Xue Jiang\thanks{School of Mathematics and Systematic Sciences, Shenyang Normal University,  Shenyang, China
  (\email{littledonna@163.com}).}
\and Baoxin Shang\thanks{Corresponding author, College of Science, Northeast Electric Power University, Jilin, China 
  (\email{shbxin@163.com}).}
 }

\usepackage{amsopn}

	\usepackage{blkarray}
	\usepackage{subfigure}
	\usepackage{upgreek}

	\newtheorem{example}{Example}

	\DeclareMathOperator{\LM}{LM}
	\DeclareMathOperator{\lm}{lm}	
	
	\DeclareMathOperator{\e}{e}
	\DeclareMathOperator{\grlex}{grlex}

	\newcommand{\numfld}[1]{\mathbb{#1}}
	\newcommand{\Z}{\numfld{Z}}
	
	\newcommand{\F}{\numfld{F}}
	\newcommand{\R}{\numfld{R}}
	\newcommand{\C}{\numfld{C}}

\ifpdf
\hypersetup{
  pdftitle={Computing monomial interpolating basis for multivariate polynomial interpolation},
  pdfauthor={Y. H. Gong, X. Jiang, and B. X. Shang}
}
\fi

\begin{document}

\maketitle

\begin{abstract}
  In this paper, we study how to quickly compute the $\prec$-minimal monomial interpolating basis for a multivariate polynomial interpolation problem. We address the notion of ``reverse" reduced basis of linearly independent polynomials and design an algorithm for it. Based on the notion, for any monomial ordering we present a new method to read off the $\prec$-minimal monomial interpolating basis from monomials appearing in the polynomials representing the interpolation conditions.
\end{abstract}



\section{Introduction}
Let $\F$ be either the real field $\R$ or the complex field $\C$. Polynomial interpolation is to construct a polynomial $g$ belonging to a finite-dimensional subspace of $\F[{\bf X}]$ from a set of data that agrees with a given function $f$ at the data set, where $\F[{\bf X}] := \F[x_1, x_2, \dots, x_d]$ denotes the polynomial ring in $d$ variables over the field $\F$.

Univariate polynomial interpolation has a well-developed theory, while the multivariate one is very problematic since a multivariate polynomial interpolation  problem is determined not only by the number of the interpolation points, but also by the geometry of the data set, see \cite{Birkhoff1979, gasca1982on, gasca2000on}.

In the multivariate polynomial interpolation theory, one of the most important targets is  to find an interpolating basis, usually of minimal degree. As is well known, the most significant milestone of computation of interpolating basis is the algorithm called the MMM's algorithm, which is presented in \cite{MMM1993} by M. G. Marinari, H. M. M\"{o}ller and T. Mora.  For a point set $\Theta \subset \F^d$ and a fixed monomial ordering $\prec$, the MMM's algorithm yields a $\prec$-minimal monomial interpolating basis for a $d$-variate Hermite interpolation on $\Theta$. However the complexity of the algorithm sometimes limites its applications. In recent years, many authors have proposed new algorithms that can reduce the complexity, but most of them just deal with special cases \cite{CERLIENCO199573, FELSZEGHY2006663}.  Such as, one algorithm called Lex game algorithm \cite{FELSZEGHY2006663} can produce, with relatively small cost, an interpolating basis w.r.t. a (inverse) lexicographic ordering. Through a large number of examples, we observed that the polynomials, which represent the interpolation conditions, must contain some monomial interpolating basis. So we address a new method to read off a monomial interpolating basis from them. Moreover, the monomial interpolating basis is of minimal degree.  The research is based on ideal interpolation, hence it is necessary to introduce the concept of ideal interpolation.

For studying multivariate polynomial interpolation, Birkhoff \cite{Birkhoff1979} first introduces the definition of ideal interpolation.  Ideal interpolation can be defined by a linear idempotent projector whose kernel is a polynomial ideal. In ideal interpolation \cite{DeBoor2004}, the interpolation conditions at an interpolation point $\boldsymbol{\uptheta} \in \F^d$ can be described by a linear space ${\rm span}\{\delta_{\boldsymbol{\uptheta}} \circ P(D), P \in P_{\boldsymbol{\uptheta}}\}$, where $P_{\boldsymbol{\uptheta}}$ is a $D$-invariant polynomial subspace, $\delta_{\boldsymbol{\uptheta}}$ is the evaluation functional at $\boldsymbol{\uptheta}$ and $P(D)$ is the differential operator induced by $P$. Lagrange interpolation is a standard example where all $P_{\boldsymbol{\uptheta}}={\rm span}\{1\}$. Ideal interpolation provides a natural link between polynomial interpolation and algebraic geometry.

The paper is organized as follows. The notion of ``reverse" reduced basis and the algorithm are described in   \cref{sec:alg}. Based on the new algorithm, we show how to read off a monomial interpolating basis for a interpolation problem at the zero point in \cref{sec:zeropoint}. Our main results are presented in \cref{sec:severalpoints}, and the conclusions follow in
\cref{sec:conclusions}.

\section{Preliminaries}
\label{sec:Preliminaries}

Throughout the paper, $\Z_{\ge 0}$ denotes the set of nonnegative integers. Let  $\Z_{\ge 0}^d := \{ (\alpha_1, \alpha_2, \dots, \alpha_d) \mid  \alpha_i \in \Z_{\ge 0}\}$.  For  $\alpha:=(\alpha_1, \alpha_2, \dots, \alpha_d) \in \Z_{\ge 0}^d$, $\alpha!:=\alpha_1!\alpha_2!\cdots\alpha_d!$ and we write ${\bf X}^\alpha$ for the monomial $x_1^{\alpha_1} x_2^{\alpha_2} \cdots x_d^{\alpha_d}$.  A polynomial $P \in \F[{\bf X}]$ can be considered as the formal power series
\begin{displaymath}
    P=\sum_{\alpha \in \Z_{\ge 0}^d} \hat{P}(\alpha) {\bf X}^\alpha,
\end{displaymath}
where $\hat{P}(\alpha)$'s are the coefficients in the polynomial $P$.

$P(D):=P(D_{x_1}, D_{x_2}, \dots, D_{x_d})$ is the differential operator induced by the polynomial $P$, where $D_{x_j}:=\frac{\partial}{\partial x_j}$ is the differentiation with respect to the $j$th variable, $j=1, 2, \dots, d$.

Let $D^{\alpha}:=D_{x_1}^{\alpha_1}D_{x_2}^{\alpha_2}\cdots D_{x_d}^{\alpha_d}$, the differential polynomial is defined as
\begin{displaymath}
    P(D):=\sum_{\alpha\in\Z_{\ge 0}^d} \hat{P}(\alpha)D^{\alpha}.
\end{displaymath}

Given a monomial ordering $\prec$, the leading monomial of a polynomial $P \in \F[{\bf X}]$ w.r.t. $\prec$ is defined by
\begin{displaymath}
    \LM(P) := \max_{\prec} \{ {\bf X}^{\alpha} \mid \hat{P}(\alpha) \neq 0\},
\end{displaymath}
and the least monomial of  the polynomial $P$ w.r.t $\prec$ is defined by
\begin{displaymath}
    \lm (P) := \min_{\prec}  \{ {\bf X}^{\alpha} \mid \hat{P}(\alpha) \neq 0\}.
\end{displaymath}

    For example, fixing the monomial ordering $\grlex(y \prec x)$, for $P=\frac{1}{6} x^3 + x y + y \in \F[x, y]$, we have
	\begin{displaymath}
	    \LM(P) = x^3, \quad \lm (P) =y.
	\end{displaymath}

	\begin{definition} \rm
	    We denote by $\Lambda\{P_1, P_2, \dots, P_n\}$ the set of all monomials that occur in  the polynomials $P_1, P_2, \dots, P_n$ with nonzero coefficients.
	\end{definition}
	
For example, let $P_1=1, P_2=x, P_3=\frac{1}{2}x^2 +y$, then
	    \begin{displaymath}
	        \Lambda \{P_1, P_2, P_3\} = \{1, x, y, x^2\}.
	    \end{displaymath}

	\begin{definition} \rm
	    $\{P_1, P_2, \dots, P_n\} \subset \F[{\bf X}]$ is called a reduced basis, if
		\begin{enumerate}
		 \item $P_1, P_2, \dots, P_n$ are linearly independent,
		 \item $\LM(P_i) \not\in \Lambda \{P_j\}, 1 \le i < j \le n$.
		\end{enumerate}
	\end{definition}

	\begin{example} Fixing the monomial ordering $\grlex(z \prec y \prec x)$,
	    \begin{displaymath}
			\begin{array}{c}
			    \{1, x, \frac{1}{2}x^2 + y, \frac{1}{6} x^3 + xy + y\}, \\[0.1cm]
			    \{1, y+z, x\}, \\ [0.1cm]
				\{1, y+z, x+z\}
			\end{array}
	    \end{displaymath}
		are reduced bases.
	\end{example}

	\begin{definition} \rm
	    $\{P_1, P_2, \dots, P_n\} \subset \F[{\bf X}]$ is called a ``reverse" reduced basis, if
		\begin{enumerate}
		 \item $P_1, P_2, \dots, P_n$ are linearly independent,
		 \item $\lm(P_i) \not\in \Lambda \{P_j\}, 1 \le i < j \le n$.
		\end{enumerate}
	\end{definition}

	\begin{example} Fixing the monomial ordering $\grlex(z \prec y \prec x)$,
	    \begin{displaymath}
			\begin{array}{c}
			    \{1, x, \frac{1}{2}x^2 + y, \frac{1}{6} x^3 - \frac{1}{2}x^2 + xy\}, \\[0.1cm]
			    \{1, y+z, x\}, \\
				\{1, y+z, x-y\}
			\end{array}
	    \end{displaymath}
		are ``reverse" reduced bases.  	
	\end{example}

	\section{The algorithm to compute a ``reverse" reduced bases} \label{sec:alg}
	For a monomial ordering $\prec$ and linearly independent polynomials $P_1, P_2, \dots, P_n \in \F[{\bf X}]$, \cref{alg:revredbasis} computes a ``reverse" reduced basis.
	\begin{algorithm}
	\caption{A ``reverse" reduced basis}
	\label{alg:revredbasis}
	\begin{algorithmic}[1]
	\STATE{{\bf Input:} $P_1, P_2, \dots, P_n \in \F[{\bf X}]$, linearly independent polynomials}
	\STATE{\hspace*{1.1cm}$\prec$, a monomial ordering}
	\STATE{{\bf Output:} $\{P_1^{(n-1)}, P_2^{(n-1)}, \dots, P_n^{(n-1)}\}$, a ``reverse" reduced basis}	
	\STATE{//Initialization}
	\STATE{$P_1^{(0)}=P_1, P_2^{(0)}=P_2, \dots, P_n^{(0)}=P_n$;}
	\STATE{//Computing}
	\FOR{$k=1:n-1$}
	\STATE{$P_j^{(k)}=P_j^{(k-1)}, 1 \le j \le k$;}
		\FOR{$j=k+1:n$}
			\STATE{${\bf X}^{\beta_k^{(k-1)}} = \lm (P_k^{(k-1)})$;}
			\STATE{$P_j^{(k)}=P_j^{(k-1)} - \left( \frac{\hat{P}_j^{(k-1)}(\beta_k^{(k-1)})}{\hat{P}_k^{(k-1)}(\beta_k^{(k-1)})}\right) P_k^{(k-1)}$;}
		\ENDFOR
	\ENDFOR
	\RETURN $\{P_1^{(n-1)}, P_2^{(n-1)}, \dots, P_n^{(n-1)}\}$;
	\end{algorithmic}
	\end{algorithm}

	 Notice that $\hat{P}_k^{(k-1)}(\beta_k^{(k-1)})$ is the coefficient of the least monomial of ${P}_k^{(k-1)}$, so it is nonzero. It is obvious that \cref{alg:revredbasis} terminates. The following theorem shows its correctness.

	\begin{theorem}
	    A set of linearly independent polynomials can be transformed into a ``reverse" reduced basis.
	\end{theorem}
	
	\begin{proof}
	    According to Line 11 in \cref{alg:revredbasis}, it is obvious that
		\begin{equation} \abovedisplayskip=2pt \belowdisplayskip=2pt \label{eqn:thm1:proof:2}
			\lm(P_k^{(k-1)}) = {\bf X}^{\beta_k^{(k-1)}} \not\in \Lambda\{P_j^{(k)}\},\; 1 \le k < j \le n.
		\end{equation}
		Denote $b_j^{(k-1)} = \frac{\hat{P}_j^{(k-1)}(\beta_k^{(k-1)})}{\hat{P}_k^{(k-1)}(\beta_k^{(k-1)})}$, then Line 11 in \cref{alg:revredbasis} becomes
		\begin{equation} \abovedisplayskip=2pt \belowdisplayskip=2pt
		    \label{eqn:thm1:proof:3}
			P_j^{(k)} = P_j^{(k-1)} - b_j^{(k-1)}P_k^{(k-1)}, 1 \le k < j \le n.
		\end{equation}
		Next, we prove that for each $k \in \{1, 2, \dots, n-1\}$,

		1. $P_1^{(k)}, P_2^{(k)}, \dots, P_n^{(k)}$ are linearly independent,

		2. $\lm(P_i^{(k)}) \not\in \Lambda\{P_j^{(k)}\}, j=i+1, \dots, n, i=1, 2, \dots, k$.
		
		Use induction on $k$.

		If $k=1$, from Line 8 in the algorithm and \cref{eqn:thm1:proof:3}, we get
		\begin{displaymath}
		    P_1^{(1)}=P_1^{(0)}, P_2^{(1)}=P_2^{(0)} - b_2^{(0)}P_1^{(0)}, \dots, P_n^{(1)}=P_n^{(0)} - b_n^{(0)}P_1^{(0)}.
		\end{displaymath}
		Suppose that there exsits $c_1, c_2, \dots, c_n$ satisfying
		\begin{displaymath}
		    c_1 P_1^{(1)} + c_2 P_2^{(1)} + \dots + c_n P_n^{(1)} =0,
		\end{displaymath}
		then
		\begin{displaymath}
		    c_1 P_1^{(0)} + c_2(P_2^{(0)} - b_2^{(0)}P_1^{(0)}) + c_n (P_n^{(0)} - P_n^{(0)}) =0,
		\end{displaymath}
		which means
		\begin{displaymath}
		    (c_1 - c_2 b_2^{(0)} - \dots - c_n b_n^{(0)}) P_1^{(0)} + c_2 P_2^{(0)} + \dots + c_n P_n^{(0)} =0,
		\end{displaymath}
		i.e.,
		\begin{displaymath}
			(c_1 - c_2 b_2^{(0)} - \dots - c_n b_n^{(0)}) P_1 + c_2 P_2 + \dots + c_n P_n =0.
		\end{displaymath}
		Since $P_1, P_2, \dots, P_n$ are linearly independent, we get $c_1=c_2=\dots=c_n =0$. Thus $P_1^{(1)}, P_2^{(1)}, \dots, P_n^{(1)}$ are linearly independent.  From \cref{eqn:thm1:proof:2}, we have $\lm(P_1^{(0)}) \not\in \Lambda\{P_j^{(1)}\}, i=2, 3, \dots, n$.  Due to $P_1^{(1)}=P_1^{(0)}$, we get $\lm(P_1^{(1)}) \not\in \Lambda\{P_j^{(1)}\}, j=2, 3, \dots, n$. In summary, if $k=1$, the conclusions hold.
		
		Suppose that if $k=t\, (1<t < n-1)$, the conclusions hold.  It means that $P_1^{(t)}, P_2^{(t)}, \dots, P_n^{(t)}$ are linearly independent and $\lm(P_i^{(t)}) \not\in \Lambda\{P_j^{(t)}\}, j=i+1, \dots, n, i=1, 2, \dots, t$.

		If $k=t+1$,  from Line 8 in the algorithm and \cref{eqn:thm1:proof:3},  we get
		\begin{displaymath}
		    \begin{aligned} \abovedisplayskip=2pt \belowdisplayskip=2pt
		        P_1^{(t+1)} & = P_1^{(t)}, P_2^{(t+1)} = P_2^{(t)}, \dots, P_{t+1}^{(t+1)}  = P_{t+1}^{(t)}, \\
		        P_{t+2}^{(t+1)} & = P_{t+2}^{(t)} - b_{t+2}^{(t)} P_{t+1}^{(t)}, \\
				P_{t+3}^{(t+1)} & = P_{t+3}^{(t)} - b_{t+3}^{(t)} P_{t+1}^{(t)}, \\
										& \;\; \vdots \\
				P_n^{(t+1)} & = P_n^{(t)} - b_n^{(t)} P_{t+1}^{(t)}.
		    \end{aligned}
		\end{displaymath}
		Suppose that there exist $c_1, c_2, \dots, c_n$ satisfying
		\begin{displaymath}
		    c_1 P_1^{(t+1)} + c_2 P_2^{(t+1)} + \dots + c_n P_n^{(t+1)}=0,
		\end{displaymath}
		then
		\begin{displaymath}
			\begin{array}{l}
			    c_1 P_1^{(t)} + c_2 P_2^{(t)} + \dots + c_{t+1} P_{t+1}^{(t)}  \\
			     \phantom{M} + c_{t+2} (P_{t+2}^{(t)} - b_{t+2}^{(t)} P_{t+1}^{(t)}) + c_{t+3} (P_{t+3}^{(t)} - b_{t+3}^{(t)} P_{t+1}^{(t)}) + \dots + c_n (P_{n}^{(t)} - b_{n}^{(t)} P_{t+1}^{(t)}) =0,
			\end{array}		
		\end{displaymath}
		so	
		\begin{displaymath}
			\begin{array}{l}
			     c_1 P_1^{(t)} + c_2 P_2^{(t)} + \dots + c_{t} P_{t}^{(t)} \\
			      \phantom{M} + (c_{t+1} - c_{t+2} b_{t+2}^{(t)} - c_{t+3} b_{t+3}^{(t)} - \dots - c_{n} b_{n}^{(t)}) P_{t+1}^{(t)} \\
					\phantom{M}  + c_{t+2}P_{t+2}^{(t)} + c_{t+3}P_{t+3}^{(t)} + \dots + c_{n}P_{n}^{(t)} =0.
			\end{array}	
		\end{displaymath}
		According to the induction condition, $P_1^{(t)}, P_2^{(t)}, \dots, P_n^{(t)}$ are linearly independent, we arrive at $c_1 = c_2 = \dots = c_n =0$, thus $P_1^{(t+1)}, P_2^{(t+1)}, \dots, P_n^{(t+1)}$ are also linearly independent.

		If $k=t+1$, by \cref{eqn:thm1:proof:3}
		\begin{displaymath}
		    P_j^{(t+1)} = \left\{
			 \begin{array}{l}
			      P_j^{(t)}, j=1, 2, \dots, t+1, \\
			      P_j^{(t)} - b_j^{(t)}P_{t+1}^{(t)}, j=t+2, t+3, \dots, n.\\
			 \end{array}
			\right.
		\end{displaymath}
		It means that $P_j^{(t+1)}$'s are represented by $P_j^{(t)}, j=1, 2, \dots, t$, and $P_j^{(t+1)}$'s are represented by $P_j^{(t)}$ and $P_{t+1}^{(t)}$, $j=t+2, t+3, \dots, n$. Then we have
		\begin{equation} \abovedisplayskip=2pt \belowdisplayskip=2pt
			\label{eqn:thm1:proof:4}
		    \Lambda\{P_{i+1}^{(t+1)}, P_{i+2}^{(t+1)}, \dots, P_{n}^{(t+1)}\} \subseteq \Lambda\{P_{i+1}^{(t)}, P_{i+2}^{(t)}, \dots, P_{n}^{(t)}\}, i=1, 2, \dots, t.
		\end{equation}

		Due to the induction condition, $\lm(P^{(t)}_i) \not\in \Lambda\{P_j^{(t)}\}, j=i+1, \dots, n, i=1, 2, \dots, t$, so we have $\lm(P^{(t)}_i) \not\in \Lambda\{P_{i+1}^{(t)}, P_{i+2}^{(t)}, \dots, P_{n}^{(t)}\}, i=1,2, \dots, t$.

		From \cref{eqn:thm1:proof:4}, we get
		\begin{displaymath}
		    \lm(P^{(t)}_i) \not\in \Lambda\{P_{i+1}^{(t+1)}, P_{i+2}^{(t+1)}, \dots, P_{n}^{(t+1)}\}, i=1, 2, \dots, t,
		\end{displaymath}
		thus
		\begin{displaymath}
		     \lm(P^{(t)}_i)\not\in \Lambda\{P_{j}^{(t+1)}\}, j=i+1, \dots, n, i=1, 2, \dots, t.
		\end{displaymath}

		According to Line 8 in the algorithm, $P^{(t+1)}_i = P^{(t)}_i, i=1, 2, \dots, t$, so we get
		\begin{equation} \abovedisplayskip=2pt \belowdisplayskip=2pt
			\label{eqn:thm1:proof:5}
		    \lm(P^{(t+1)}_i)\not\in \Lambda\{P_{j}^{(t+1)}\}, j=i+1, \dots, n, i=1, 2, \dots, t.
		\end{equation}

		From \cref{eqn:thm1:proof:2}, $\lm(P^{(t)}_{t+1})\not\in \Lambda\{P_j^{(t+1)}\}, j=t+2, \dots, n$. According to Line 8 in the algorithm, $P_{t+1}^{(t+1)} = P_{t+1}^{(t)}$, then we have
		\begin{equation} \abovedisplayskip=2pt \belowdisplayskip=2pt
		    \label{eqn:thm1:proof:6}
			\lm(P^{(t+1)}_{t+1})\not\in \Lambda\{P_j^{(t+1)}\}, j=t+2, \dots, n.
		\end{equation}

		Due to \cref{eqn:thm1:proof:5}, \cref{eqn:thm1:proof:6}, we get
		\begin{displaymath}
		    \lm(P^{(t+1)}_{i})\not\in \Lambda\{P_j^{(t+1)}\}, j=i+1, \dots, n, i=1, 2, \dots, t+1.
		\end{displaymath}
		Hence the conclusions hold.
		
		It means $P_1^{(n-1)}, P_2^{(n-1)}, \dots, P_n^{(n-1)}$ are linearly independent and
		\begin{displaymath}
		     \lm(P^{(n-1)}_{i})\not\in \Lambda\{P_j^{(n-1)}\}, 1 \le i < j \le n,
		\end{displaymath}
		i.e., $\{P_1^{(n-1)}, P_2^{(n-1)}, \dots, P_n^{(n-1)}\}$ is a ``reverse" reduced basis.
	\end{proof}

	\begin{example}
	    Fixing the monomial ordering $\grlex(y \prec x)$. Let
		$ \{P_1, P_2, P_3, P_4\} = \{1, x, x^2+2y, \frac{1}{6} x^3 + xy + y\} \subset \F[x, y],
		$
		they are linearly independent.
		By \cref{alg:revredbasis}, we get
		\begin{displaymath}
		     \{P_1, P_2, P_3, P_4-\frac{1}{2}P_3\} = \{1, x, x^2+2y, \frac{1}{6} x^3 - \frac{1}{2}x^2 + xy\}
		\end{displaymath}
		is a ``reverse" reduced basis.
	\end{example}

	\begin{remark}
	    Actually, the algorithm is also workable for formal power series.
	\end{remark}

 \section{The interpolation problem at point zero} \label{sec:zeropoint}
	
	\begin{definition}
	Let $T$ and $T'$ be two sets of monomials in $\F[{\bf X}]$ with $T-T'\neq \emptyset$ and $T'-T\neq \emptyset$. For a monomial ordering $\prec$, we call $T'\prec T$, if
	\begin{displaymath} \abovedisplayskip=2pt \belowdisplayskip=2pt
	    \max_{\prec} (T'-T)\prec \max_{\prec} (T-T').
	\end{displaymath}
	\end{definition}

	    For example, let $ T=\{1,y,xy,x^2\}$, $ T'=\{1,x,y^2,x^2\}$. Then  $T'-T=\{x,y^2\}$,  $T-T'=\{y,xy\}$. For the monomial ordering $\grlex(y \prec x)$, we have
		\begin{displaymath} \abovedisplayskip=2pt \belowdisplayskip=2pt
		    \max_{\prec} (T'-T)=y^2\prec xy=\max_{\prec} (T-T'),
		\end{displaymath}
		and it means $T'\prec T$.

\begin{definition}[$\prec$-minimal monomial interpolating basis]
	    Given interpolation conditions $\Delta$ and a monomial ordering $\prec$, let $T$ be a monomial interpolating basis for $\Delta$. Then $T$ is $\prec$-minimal if there exists no monomial interpolating bassis $T'$ for $\Delta$ satisfying $T' \prec T$.
	\end{definition}
	\begin{theorem}[Existence of a monomial interpolating basis \cite{SauerT1998, IdealVarietyAlgorithm2007}] \rm Given interpolation conditions \label{thm:MonBasExist}
	$
	    \Delta :=\delta_{\bf 0} \circ \{P_1(D), P_2(D), \dots, P_n(D)\}
	$
	and a monomial ordering $\prec$, there exists a unique $\prec$-minimal monomial interpolating basis for $\Delta$.
	\end{theorem}	

	\begin{lemma} \label{lem:InterpBasNonSing}
	     Given interpolation conditions $\Delta :=\delta_{\bf 0} \circ \{P_1(D), P_2(D), \dots, P_n(D)\}$  and a set of monomials $T:=\{{\bf X}^{\beta_1}, {\bf X}^{\beta_2}, \dots, {\bf X}^{\beta_n}\}$, the matrix applying $T$ on $\Delta$ is denoted by
		\begin{displaymath}
		    T_{\Delta} :=(\delta_{\bf 0} \circ P_i(D) {\bf X}^{\beta_j})_{ij}, \quad 1 \le i, j \le n.
		\end{displaymath}
		Then $T$ is a monomial interpolating basis iff $T_{\Delta}$ is non-singular. 		
	\end{lemma}

	\begin{proof} \rm
	    Suppose that the interpolating polynomial $g=\sum_{j=1}^n c_j {\bf X}^{\beta_j}$ and the values are $f_i$'s, $i=1, 2, \dots, n$. It means
		\begin{displaymath}
		    (\delta_{\bf 0} \circ P_i(D)) g = f_i, i=1, 2, \dots, n. 
		\end{displaymath}
		Thus we get the linear equations
		\begin{displaymath}
		  \begin{blockarray}{ccccc} 
		         & {\bf X}^{\beta_1}  & {\bf X}^{\beta_2} & \cdots & {\bf X}^{\beta_n}  \\
		      \begin{block}{r(cccc)}
		          \delta_{\bf 0} \circ P_1(D) & \delta_{\bf 0} \circ P_1(D) {\bf X}^{\beta_1} & \delta_{\bf 0} \circ P_1(D) {\bf X}^{\beta_2} &  \cdots & \delta_{\bf 0} \circ P_1(D) {\bf X}^{\beta_n}\\
				  \delta_{\bf 0} \circ P_2(D) & \delta_{\bf 0} \circ P_2(D) {\bf X}^{\beta_1} & \delta_{\bf 0} \circ P_2(D) {\bf X}^{\beta_2} &  \cdots & \delta_{\bf 0} \circ P_2(D) {\bf X}^{\beta_n}\\
				 \vdots\phantom{abcd} & \vdots  & \vdots   & \vdots & \vdots  \\
				 \delta_{\bf 0} \circ P_n(D) & \delta_{\bf 0} \circ P_n(D) {\bf X}^{\beta_1} & \delta_{\bf 0} \circ P_n(D) {\bf X}^{\beta_2} &  \cdots & \delta_{\bf 0} \circ P_n(D) {\bf X}^{\beta_n}\\
		      \end{block}
		  \end{blockarray}\begin{pmatrix}
		      c_1 \\ c_2 \\ \vdots \\ c_n\\
		    \end{pmatrix} =\begin{pmatrix}
		      f_1 \\ f_2 \\ \vdots \\ f_n\\
		    \end{pmatrix}.
		\end{displaymath}
		So the coefficient matrix $T_{\Delta}$ is non-singular $\Leftrightarrow$ The linear equations has a unique solution.
	\end{proof}

	\begin{example} \rm
	  Given interpolation conditions  $\Delta=\delta_{\bf 0} \circ \{1, D_y + D_z, D_x\}$ and a  set of monomials $T= \{1, z, x\}$, then
		\begin{displaymath}
		    T_\Delta =
			\begin{blockarray}{cccl} 
			       1 & z &x  &  \\
			    \begin{block}{(ccc)l}
			         1& 0 & 0 & \delta_{\bf 0} \circ 1 \\
					 0 & 1 & 0 & \delta_{\bf 0} \circ (D_y + D_z) \\
					 0 & 0 & 1 & \delta_{\bf 0} \circ D_x \\
			    \end{block}
			\end{blockarray}.
		\end{displaymath}
		It is obvious that $T_{\Delta}$ is non-singular, so $T$ is a monomial interpolating basis for $\Delta$.
	\end{example}
	
	\begin{lemma} \label{lem:MonBasSubsetInterpMon}
	    Given interpolation conditions $\Delta = \delta_{\bf 0} \circ \{P_1(D), P_2(D), \dots, P_n(D)\}$, let $T:=\{{\bf X}^{\beta_1}, {\bf X}^{\beta_2}, \dots, {\bf X}^{\beta_n}\}$ be a monomial interpolating basis for $\Delta$, then $T \subseteq \Lambda\{P_1, P_2, \dots, P_n\}$.
	\end{lemma}

	\begin{proof} \rm
	    We will prove this by contradiction.  Suppose that there exists ${\bf X}^{\beta} \in T$ with ${\bf X}^{\beta} \not \in \Lambda\{P_1, P_2, \dots, P_n\}$. From
		\begin{displaymath}
		    \delta_{\bf 0} \circ D^{\alpha} {\bf X}^{\beta} =
			\begin{cases}
			   \beta!, & \alpha=\beta, \\
			   0,   & \alpha \neq \beta, \\
			\end{cases}
		\end{displaymath}
		we have $[\delta_{\bf 0} \circ P_i(D)] {\bf X}^{\beta} = [\delta_{\bf 0} \circ \sum \hat{P_i}(\alpha) D^{\alpha}] {\bf X}^{\beta} =\sum \hat{P_i}(\alpha) \underbrace{(\delta_{\bf 0} \circ D^{\alpha} {\bf X}^{\beta})}_{0} = 0, 1 \le i \le n$.  It means the matrix $T_{\Delta}$ has a zero column, so $T_{\Delta}$ is singular. It contradicts with \cref{lem:InterpBasNonSing}.
	\end{proof}

	\begin{remark} \rm
	    \cref{lem:MonBasSubsetInterpMon} implies that the monomial interpolating basis must be selected from the set of monomials appearing in the interpolation conditions.
	\end{remark}
	
	\begin{lemma} \label{lem:InterpCondHasMon}
	    Given interpolation conditions $\Delta = \delta_{\bf 0} \{P_1(D), P_2(D), \dots, P_n(D)\}$, let $T=\{{\bf X}^{\beta_1}, {\bf X}^{\beta_2}, \dots, {\bf X}^{\beta_n}\}$ be a monomial interpolating basis for $\Delta$. Then for each $ P_i, 1 \le i \le n$, there exists ${\bf X}^{\alpha_i} \in \Lambda\{P_i\}$ satisfying ${\bf X}^{\alpha_i} \in T$.
	\end{lemma}

	\begin{proof}
	    We will prove this by contradiction. Without loss of generality, we can assume that for every ${\bf X}^{\alpha} \in \Lambda\{P_1\}$, ${\bf X}^{\alpha} \not\in T$. Then it is observed that
		\begin{displaymath}
		    \begin{aligned} \abovedisplayskip=2pt \belowdisplayskip=2pt
		        [\delta_{\bf 0} \circ P_1(D)] {\bf X}^{\beta_j} & = [\delta_{\bf 0} \circ \sum \hat{P_1}(\alpha) D^{\alpha}] {\bf X}^{\beta_j} \\
		           &= \sum \hat{P_1}(\alpha) \underbrace{(\delta_{\bf 0} \circ D^{\alpha} {\bf X}^{\beta_j})}_{0} \\
				   &= 0, \quad 1 \le j \le n. 
		    \end{aligned}
		\end{displaymath}
So the matrix $T_{\Delta}$ has a zero row, and it is singuar. It contradicts with \cref{lem:InterpBasNonSing}.
	\end{proof}

	\begin{remark}
	    Given interpolation conditions $\Delta = \delta_{\bf 0} \circ \{P_1(D), P_2(D), \dots, P_n(D)\}$, \cref{lem:InterpCondHasMon} shows that each  $P_i$ contains at least a monomial in the monomial interpolating basis.
    \end{remark}

	\begin{theorem} \label{thm:maintheorem}
	    Given interpolation conditions $\Delta = \delta_{\bf 0} \circ \{P_1(D)$, $P_2(D), \dots, P_n(D)\}$,  if $\{P_1, P_2, \dots, P_n\}$ is a ``reverse" reduced basis w.r.t. a monomial ordering $\prec$, then $\{\lm(P_1), \lm(P_2), \dots, \lm(P_n)\}$ is the $\prec$-minimal monomial interpolating basis.
	\end{theorem}

	\begin{proof}
	    According to \cref{lem:InterpCondHasMon}, we must choose at least one monomial from $P_i$ to form the monomial interpolating basis. It is obvious that $\{\lm(P_1), \lm(P_2), \dots, \lm(P_n)\}$ we choose is of minimal degree. Thus we only need to prove $T=\{\lm(P_1), \lm(P_2), \dots, \lm(P_n)\}$ do construct a monomial basis, i.e., $T_\Delta$ is non-singular.

		Let $\lm(P_i) = {\bf X}^{\beta_i}$ and $P_i =\sum \hat{P_i}(\alpha) {\bf X}^{\alpha} + \hat{P_i}(\beta_i) {\bf X}^{\beta_i}, 1 \le i \le n$. Since $\{P_1, P_2, \dots, P_n\}$ is a ``reverse" reduced basis, it means
		\begin{displaymath}
		    \lm(P_i) \not\in \Lambda \{P_j\}, \quad 1 \le i < j \le n.
		\end{displaymath}
		Hence we have
		\begin{displaymath}
		    (\delta_{\bf 0} \circ P_j(D)) (\lm (P_i)) =
			\begin{cases}
			   0, & i<j,  \\
			    \beta_i ! \hat{P_j} (\beta_i)=\beta_i ! \hat{P_i} (\beta_i) \neq 0, & i=j,  \\
			\end{cases} \quad 1 \le i, j \le n. 
		\end{displaymath}
		So $T_{\Delta}$ is a upper-triangular matrix with diagonal elements nonzero, i.e., it is non-singular.
	\end{proof}

	\begin{example} \rm
	    Given interpolation conditions
		\begin{displaymath} \abovedisplayskip=2pt \belowdisplayskip=2pt
		    \delta_{\bf 0} \circ \{P_1(D), P_2(D), P_3(D)\}=\delta_{\bf 0} \circ \{1, D_y + D_z, D_x\},
		\end{displaymath}
		it is easy to see $\{P_1, P_2, P_3\}= \{1, y+z, x\}$ is a ``reverse" reduced basis w.r.t. $\grlex(z \prec y \prec x)$. Then by  \cref{thm:maintheorem} we know that $\{\lm(P_1), \lm(P_2), \lm(P_3)\} = \{1, z, x\}$  is the $\prec$-minimal monomial interpolating basis.
	\end{example}

	\begin{example}
	     Given interpolation conditions
		 \begin{displaymath} \abovedisplayskip=2pt \belowdisplayskip=2pt
		     \delta_{\bf 0} \circ \{1, D_x, \frac{1}{2} D_x^2 + D_y, \frac{1}{6} D_x^3 - \frac{1}{2} D_x^2 + D_x D_y\},
		 \end{displaymath}
		it is easy to see $\{1, x, \frac{1}{2}x^2 + y, \frac{1}{6} x^3 - \frac{1}{2} x^2 + xy\}$ is a ``reverse" reduced basis w.r.t. $\grlex(y \prec x)$.				Then by  \cref{thm:maintheorem} we know that $\{1, x, y, xy\}$ is the $\prec$-minimal monomial interpolating basis.
	\end{example}

	\section{The interpolation problem at several points} \label{sec:severalpoints}

	Let ${\bf X}=(x_1, x_2, \dots, x_d)$ and  $\F[[{\bf X}]]$ be the ring of formal power series .  For $\boldsymbol{\uptheta}=(\theta_1, \theta_2, \dots, \theta_d) \in \F^d$, we denote by $\boldsymbol{\uptheta} {\bf X} :=\sum_{i=1}^d \theta_i x_i$. From Taylor's formula, we have
	\begin{displaymath}
		\e^{\boldsymbol{\uptheta} {\bf X}} = \sum_{j=0}^\infty \frac{(\boldsymbol{\uptheta} {\bf X})^j}{j!}.
	\end{displaymath}
	Furthermore, since $\F[{\bf X}]$ is isomorphic to $\F[[{\bf X}]]$, by \cite{deboor1990on} we get
	\begin{equation} \label{eqn:nonzeroCons2zeroCons}
	    \delta_{\boldsymbol{\uptheta}} \circ \{P_1(D), P_2(D), \dots, P_n(D)\} =
		\delta_{\bf 0} \circ \{ \e^{\boldsymbol{\uptheta} D} P_1(D), \e^{\boldsymbol{\uptheta} D} P_2(D), \dots, \e^{\boldsymbol{\uptheta} D} P_n(D)\}.
	\end{equation}
	It means that an interpolation problem at a nonzero point can be converted into one at the zero point.

	\begin{theorem}[Main theorem] \label{thm:sevptsbasis}
	    For a monomial ordering $\prec$ and interpolation conditions
		\begin{displaymath}
		    \Delta=
			\left\{
				\begin{array}{l}
				    \delta_{\boldsymbol{\uptheta}_1} \circ \{P_{11}(D), P_{12}(D), \dots, P_{1s_1}(D)\} \\
				    \delta_{\boldsymbol{\uptheta}_2} \circ \{P_{21}(D), P_{22}(D), \dots, P_{2s_2}(D)\} \\
					\hspace*{2.5cm} \vdots\\
					\delta_{\boldsymbol{\uptheta}_m} \circ \{P_{m1}(D), P_{m2}(D), \dots, P_{ms_m}(D)\} \\
				\end{array}
			\right.
		\end{displaymath}
		where $\boldsymbol{\uptheta}_i \in \F^d, i=1, 2, \dots, m$, let $\{q_{ij} \;|\; i=1, 2, \dots, m, j=1, 2, \dots, s_i\}$ be a ``reverse" reduced basis of $\{ \e^{\boldsymbol{\uptheta}_i {\bf X}}P_{ij} \;|\; i=1, 2, \dots, m, j=1, 2, \dots, s_i\}$. Then $\{ \lm (q_{ij}) \;|\; i=1, 2, \dots, m, j=1, 2, \dots, s_i\}$ is the $\prec$-minimal monomial interpolating basis.
	\end{theorem}

	\begin{proof} \rm
	    From \cref{eqn:nonzeroCons2zeroCons}, the interpolation conditions are equivalent to
		\begin{displaymath}
		    \Delta=
			\left\{
				\begin{array}{l}
				    \delta_{0} \circ \{\e^{\boldsymbol{\uptheta}_1D} P_{11}(D), \e^{\boldsymbol{\uptheta}_1D} P_{12}(D), \dots, \e^{\boldsymbol{\uptheta}_1D} P_{1s_1}(D)\} \\
				    \delta_{0} \circ \{\e^{\boldsymbol{\uptheta}_2 D} P_{21}(D), \e^{\boldsymbol{\uptheta}_2 D} P_{22}(D), \dots, \e^{\boldsymbol{\uptheta}_2 D} P_{2s_2}(D)\} \\
					\hspace*{2.5cm} \vdots\\
					\delta_{0} \circ \{\e^{\boldsymbol{\uptheta}_m D} P_{m1}(D), \e^{\boldsymbol{\uptheta}_m D} P_{m2}(D), \dots, \e^{\boldsymbol{\uptheta}_m D} P_{ms_m}(D)\} \\
				\end{array}
			\right.
		\end{displaymath}
		It is an interpolation problem at the zero point. Note that $P_{ij}(D), j=1, 2, \dots, s_i$, are linearly independent, $i=1, 2, \dots, m$, so $\e^{\boldsymbol{\uptheta}_i D} P_{ij}(D), j=1, 2, \dots, s_i, i=1, 2, \dots, m$, are also linearly independent. Using \cref{alg:revredbasis}, we can compute a ``reverse" reduced basis  $\{q_{ij} \;|\; j=1, 2, \dots, s_i, i=1, 2, \dots, m\}$ for $\{\e^{\boldsymbol{\uptheta}_i {\bf X}} P_{ij} \;|\; j=1, 2, \dots, s_i, i=1, 2, \dots, m\}$. Due to \cref{thm:maintheorem}, $\{ \lm (q_{ij}) \;|\; j=1, 2, \dots, s_i, i=1, 2, \dots, m \}$ is the $\prec$-minimal monomial interpolating basis.
	\end{proof}

	\begin{remark} \label{rem:finteterm}
	    Note that only the least monomial of each polynomial in a ``reverse" reduced basis is concerned, so in practical computation we only need the first finite terms of the Taylor's expansion of $\e^{\boldsymbol{\uptheta} {\bf X}}$.  For an interpolation problem with $n$ interpolation conditions, since the monomials in the $\prec$-minimal monomial interpolating basis form a lower set, we only need to compute the first finite terms of $\e^{\boldsymbol{\uptheta} {\bf X}}$ with degrees $\le n-1$ in the worst case.
	\end{remark}
	
    Next, we consider two bivariate examples.

	\begin{example}[Lagrange polynomial interpolation]
		
		\begin{figure}[htbp!] \setlength{\abovecaptionskip}{0pt} \setlength{\belowcaptionskip}{0pt}		
		 \centering
		 \subfigure[Interpolation points]{
		  \includegraphics[width=0.35\textwidth]{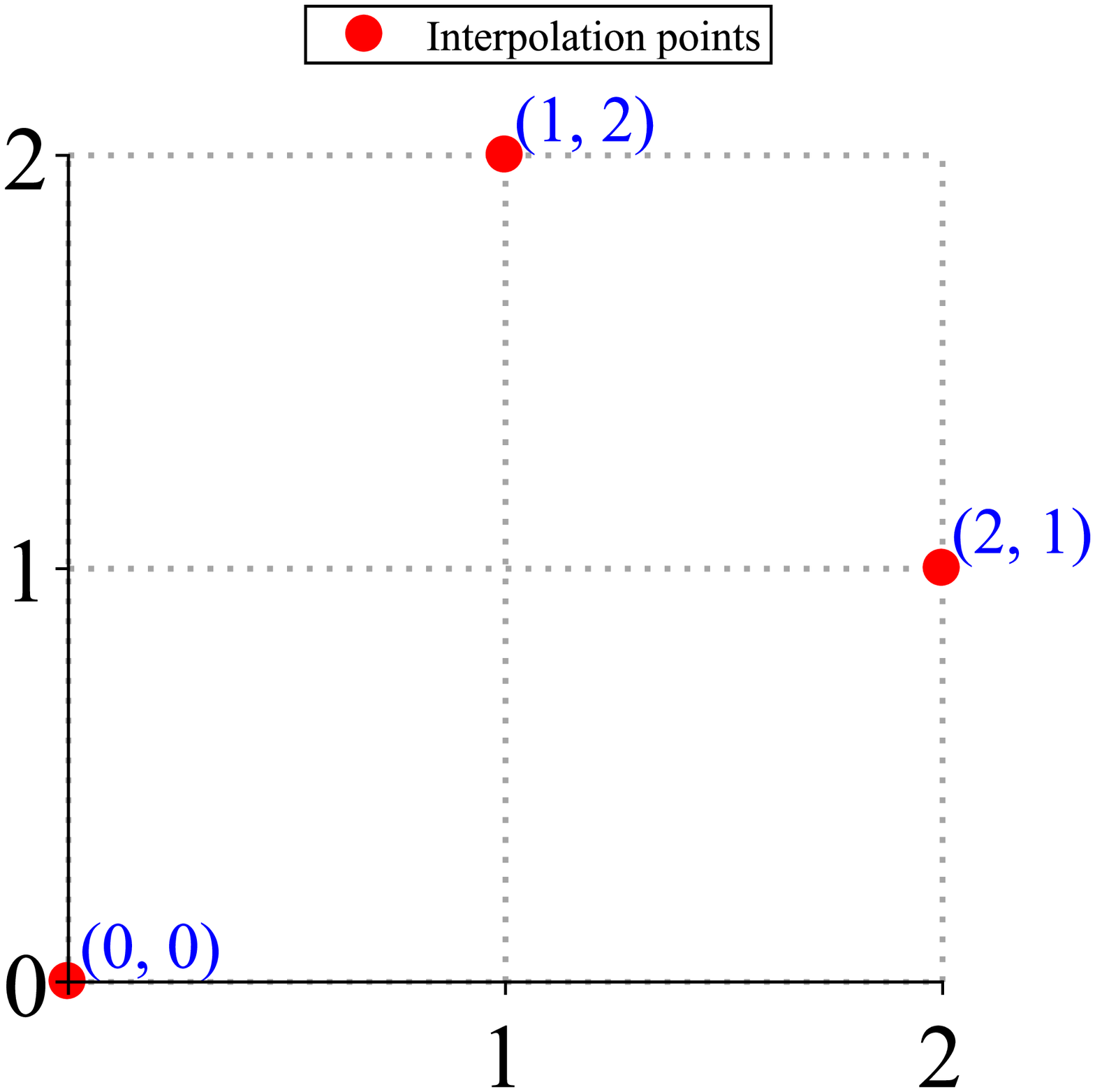}
		  \label{fig:subfig:Lag1}
		 }
		 \hspace{0.1\textwidth}
		  \subfigure[Monomial interpolating basis]{
		  \includegraphics[width=0.35\textwidth]{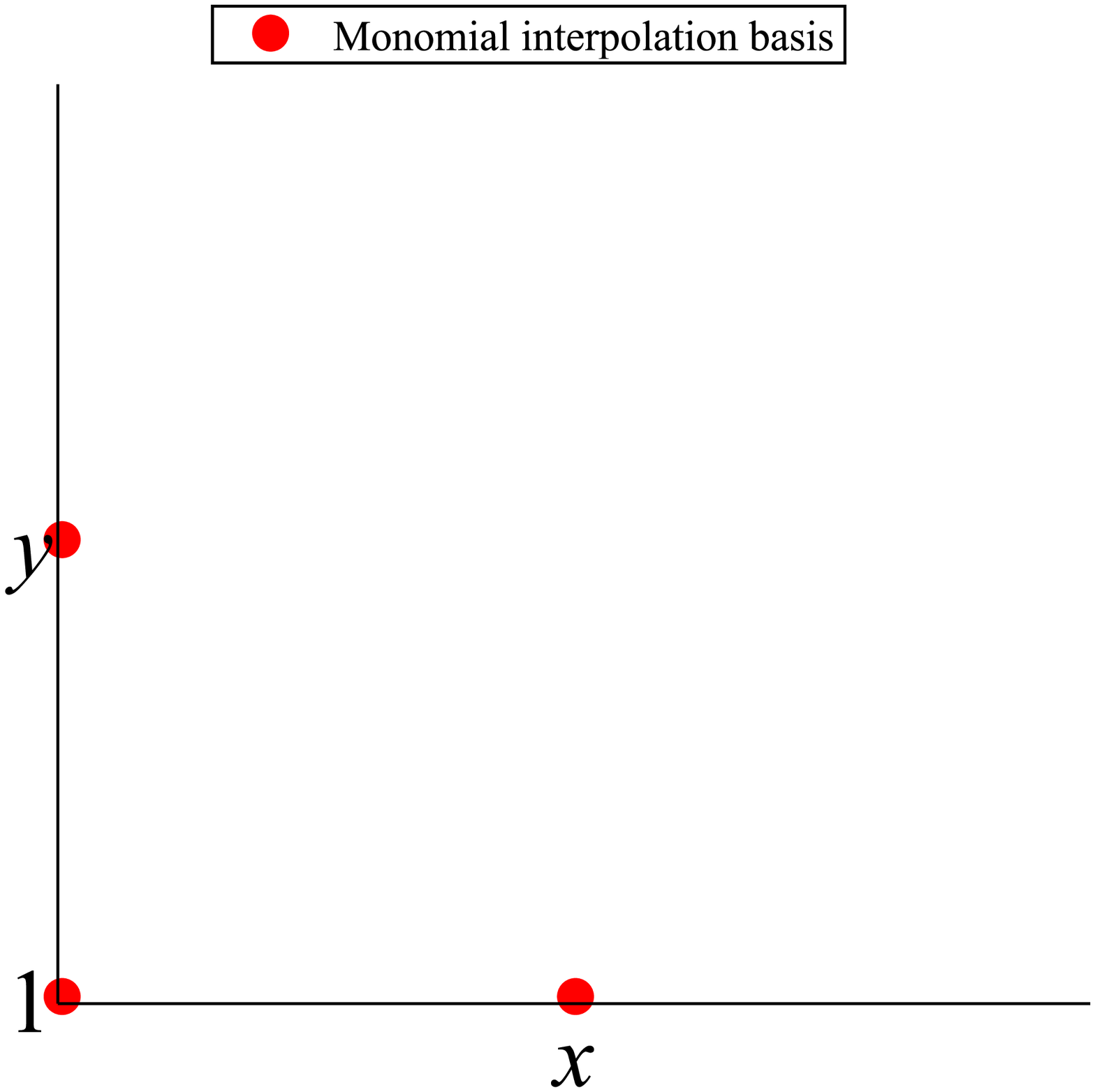}
		  \label{fig:subfig::Lag1}
		 }
		 \caption{Lagrange polynomial interpolation}
		 \label{fig:lag}
		\end{figure}
	
		 For the monomial ordering $ \grlex(y \prec x)$ and interpolation conditions
			\begin{displaymath}
			    \Delta:= \left\{
					\begin{aligned} \abovedisplayskip=2pt \belowdisplayskip=2pt
					    \delta_{(0, 0)} \circ \{1\} \\
					    \delta_{(1, 2)} \circ \{1\} \\
						\delta_{(2, 1)} \circ \{1\}
					\end{aligned}
				\right.,
			\end{displaymath}
		by \cref{eqn:nonzeroCons2zeroCons} and \cref{rem:finteterm},  we only need to compute the terms of $\e^{\boldsymbol{\uptheta} {\bf X}}$ with degrees $\le 2$, thus we have
			\begin{displaymath}
			    \begin{aligned} \abovedisplayskip=2pt \belowdisplayskip=2pt
			        \e^{(1, 2){\bf X}} & = 1 + (x+2y) + \frac{1}{2}(x^2 + 4xy + 4y^2) + \cdots, \\
			        \e^{(2, 1){\bf X}} & = 1 + (2x+y) + \frac{1}{2}(4x^2 + 4xy + y^2) + \cdots. \\
			    \end{aligned}
			\end{displaymath}		
		By \cref{thm:sevptsbasis}, we get
			\begin{displaymath}
				\begin{aligned}
				    \{p_1, p_2, p_3\}  & = \{1,  \e^{(1, 2){\bf X}} \cdot 1,  \e^{(2, 1){\bf X}} \cdot 1\}\\
				       & = \{1, \cdots + \frac{1}{2}(x^2+4xy+4y^2) +(x+2y)+1, \\
					   & \phantom{M} \cdots + \frac{1}{2}(4x^2+4xy+y^2) +(2x+y)+1\}.
				\end{aligned}
			\end{displaymath}
		By \cref{alg:revredbasis} we get a ``reverse" reduced basis $\{q_1, q_2, q_3\}=\{1, \cdots + \frac{1}{2} x^2 + 2xy + 2y^2 + x + 2y, \dots + \frac{7}{4}x^2 + xy -\frac{1}{2}y^2 + \frac{3}{2}x\}$.

			At last, by \cref{thm:sevptsbasis} we get the $\prec$-minimal monomial interpolating basis $T=\{\lm(q_1), \lm(q_2), \lm(q_3)\}=\{1, y, x\}$.  It is easy to verify that
			\begin{displaymath}
			    T_\Delta=\begin{blockarray}{cccl} 
			           1 & y & x &  \\
			        \begin{block}{(ccc)l}
			            1 & 0 & 0 & \delta_{(0, 0)} \circ \{1\} \\
						1 & 2 & 1 & \delta_{(1, 2)} \circ \{1\} \\
						1 & 1 & 2 & \delta_{(2, 1)} \circ \{1\} \\
			        \end{block}
			    \end{blockarray}.
			\end{displaymath} It is obvious that $T_{\Delta}$ is non-singular. So $T$ is indeed an interpolating basis for $\Delta$.
	\end{example}

	\begin{example}[Hermite polynomial interpolation]
	    \begin{figure}[htbp!] \setlength{\abovecaptionskip}{0pt} \setlength{\belowcaptionskip}{0pt}		
		 \centering
		 \subfigure[Interpolation points]{
		  \includegraphics[width=0.35\textwidth]{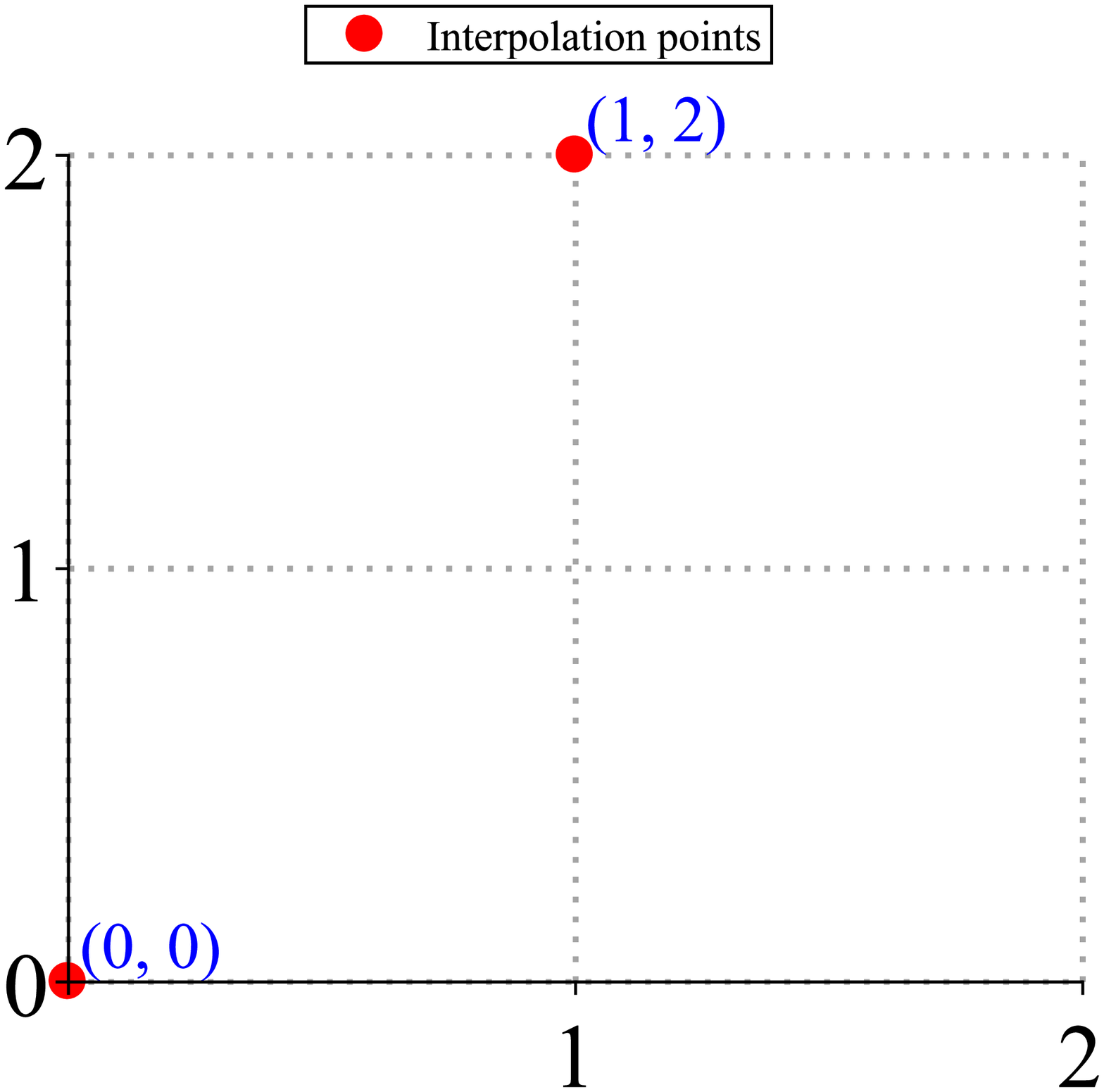}
		  \label{fig:subfig:Her1}
		 }
		 \hspace{0.1\textwidth}
		  \subfigure[Monomial interpolating basis]{
		  \includegraphics[width=0.35\textwidth]{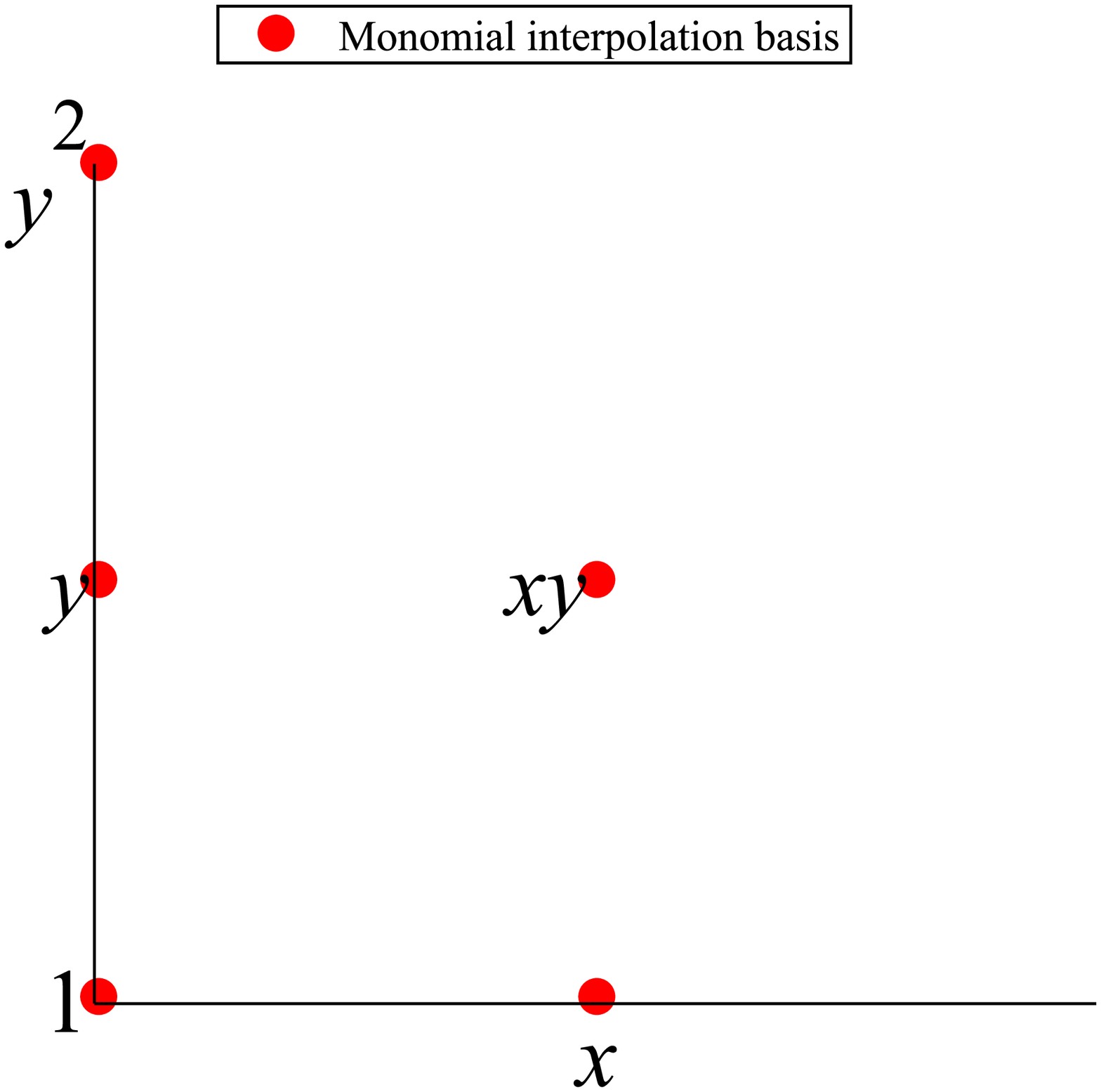}
		  \label{fig:subfig:Her2}
		 }
		 \caption{Hermite polynomial interpolation}
		 \label{fig:subfig:Her}
		\end{figure}
	
		For the monomial ordering $\grlex(y \prec x)$ and interpolation conditions
			\begin{displaymath}
			    \Delta:= \left\{
					\begin{array}{l} \abovedisplayskip=2pt \belowdisplayskip=2pt
					    \delta_{(0, 0)} \circ \{1, D_x, \frac{1}{2}D_x^2+D_y\} \\
					    \delta_{(1, 2)} \circ \{1, D_x\}						
					\end{array},
				\right.
			\end{displaymath}
			by \cref{eqn:nonzeroCons2zeroCons} and \cref{rem:finteterm},  we  compute the terms of $\e^{\boldsymbol{\uptheta} {\bf X}}$ with degrees $\le 2$, thus we have
			\begin{displaymath}			
			        \e^{(1, 2){\bf X}} = 1 + (x+2y) + \frac{1}{2}(x^2 + 4xy + 4y^2) + \cdots.	
			\end{displaymath}
			By \cref{thm:sevptsbasis}, we get
			\begin{displaymath}
			    \begin{aligned} \abovedisplayskip=2pt \belowdisplayskip=2pt
			        \{P_1, P_2, P_3, P_4, P_5\} &= \{1, x, \frac{1}{2}x^2+y, \e^{(1, 2){\bf X}} \cdot 1, \e^{(1, 2){\bf X}} \cdot x\} 	\\
										&= \{1, x, \frac{1}{2}x^2+y,	\\		
										& \phantom{D1}\cdots+\frac{1}{2}(x^2+4xy+4y^2) + (x+2y) +1, \\
										& \phantom{D1}\cdots+\frac{1}{2}(x^3+4x^2y+4xy^2) + (x^2+2xy) +x\}.
			    \end{aligned}
			\end{displaymath}
			By \cref{alg:revredbasis}, we get a ``reverse" reduced basis $\{q_1, q_2, q_3, q_4, q_5\}=\{1, x, \frac{1}{2}x^2+y, \cdots+\frac{1}{2}(-x^2+4xy+4y^2), \cdots+\frac{1}{2}(x^3+4x^2y+4xy^2) + (x^2+2xy)\}$.

			At last, by \cref{thm:sevptsbasis} we get the $\prec$-minimal monomial interpolating basis $T=\{\lm(q_1), \lm(q_2), \lm(q_3), \lm(q_4), \lm(q_5)\}=\{1, x, y,  y^2, xy\}$.  It is easy to verify that
			\begin{displaymath}
			    T_\Delta=\begin{blockarray}{cccccl} 
			           1 & x & y & y^2 & xy \\
			        \begin{block}{(ccccc)l}
			            1 &   0 &0 & 0 & 0& \delta_{(0, 0)} \circ \{1\} \\
						0 &   1 &0 & 0 & 0& \delta_{(0, 0)} \circ \{D_x\} \\
						0 &   0 &1 & 0 & 0& \delta_{(0, 0)} \circ \{\frac{1}{2}D_x^2+D_y\} \\		
						1 &   1 &2 & 4 & 2& \delta_{(1, 2)} \circ \{1 \} \\
						0 &   1 &0 & 0 & 2& \delta_{(1, 2)} \circ \{D_x \}\\
			        \end{block}
			    \end{blockarray}.
			\end{displaymath}
			It is obvious that $T_{\Delta}$ is non-singular. So $T$ is indeed an interpolating basis for $\Delta$.
	\end{example}
	
	



	
	


\section{Conclusions}
\label{sec:conclusions}

For a multivariate polynomial interpolation problem, by the concept of ``reverse" reduced basis, for any given monomial ordering $\prec$, we present a new method  to read off the $\prec$-minimal monomial interpolating basis from the set of monomials appearing in the interpolation conditions. Our algorithm only uses linear eliminating, so it has a good performance.



\bibliographystyle{siamplain}
\bibliography{PolymonialSystemSolving}
\end{document}